\documentclass[12pt]{amsart}

\newtheorem{theorem}{Theorem}[section]
\newtheorem{lemma}[theorem]{Lemma}
\newtheorem{proposition}[theorem]{Proposition}

\usepackage{amssymb}

\newcommand{\C}{{\mathbb C} }

\newcommand{\cL}{{\mathcal L} }
\newcommand{\cM}{{\mathcal M} }
\newcommand{\cO}{{\mathcal O} }
\newcommand{\cX}{{\mathcal X} }
\newcommand{\wt}{\widetilde}
\newcommand{\pt}{\partial}

\def\ol#1{{\overline{#1}}}

\numberwithin{equation}{section}

\begin{document}

\baselineskip=15pt

\title{Deligne pairing and Quillen metric}

\author[I. Biswas]{Indranil Biswas}

\address{School of Mathematics, Tata Institute of Fundamental
Research, Homi Bhabha Road, Bombay 400005, India}

\email{indranil@math.tifr.res.in}

\author[G. Schumacher]{Georg Schumacher}

\address{Fachbereich Mathematik und Informatik,
Philipps-Universit\"at Marburg, Lahnberge, Hans-Meerwein-Strasse, D-35032
Marburg, Germany}

\email{schumac@mathematik.uni-marburg.de}



\date{}

\begin{abstract}
Let $X \,\longrightarrow\, S$ be a smooth projective surjective morphism
of relative dimension $n$, where $X$ and $S$ are integral schemes over
$\mathbb C$. Let $L\longrightarrow X$ be a relatively very ample line
bundle. For every sufficiently large positive integer $m$, there is a
canonical isomorphism of the Deligne pairing $\langle L\, ,\cdots\,
,L\rangle\longrightarrow S$ with the determinant line bundle ${\rm
Det}((L- {\mathcal O}_{X})^{\otimes (n+1)}\otimes L^{\otimes m})$
\cite{PRS}. If we fix a hermitian structure on $L$ and a relative
K\"ahler form on $X$, then each of the line bundles ${\rm Det}((L-
{\mathcal O}_{X})^{\otimes (n+1)}\otimes L^{\otimes m})$ and $\langle L\,
,\cdots\, ,L\rangle$ carries a distinguished hermitian structure. We prove
that the above mentioned isomorphism between $\langle L\, ,\cdots\,
,L\rangle\longrightarrow S$ and ${\rm
Det}((L- {\mathcal O}_{X})^{\otimes (n+1)}\otimes L^{\otimes m})$
is compatible with these hermitian
structures. This holds also for the isomorphism in \cite{BSW}
between a Deligne paring and a certain determinant line bundle.
\end{abstract}

\maketitle

\section{Introduction}

Let $S$ and $X$ be integral schemes over $\mathbb C$, and let
$$
f\, :\, X \,\longrightarrow\, S
$$
be a flat projective surjective morphism. Let $n$ be the dimension of the
fibers of $f$. Take algebraic line bundles $L_0\, , L_1\, , \cdots \, ,
L_{n-1}\, , L_{n}$ over $X$. Their \textit{Deligne pairing} is an
algebraic line bundle
$$
\langle L_0\, , \cdots\, ,L_{n}\rangle\, \longrightarrow\, S
$$
\cite{Zh}, \cite{De}. If we equip the line bundles $L_0\, , L_1\, , \cdots \, ,
L_{n}$ with hermitian structure, then
$\langle L_0\, , \cdots\, ,L_{n}\rangle$ gets a hermitian structure
\cite[Section~1.2]{Zh}, \cite{De}.

For any locally free sheaf $F$ on $X$, we have the determinant line bundle
$$
{\rm Det}(F)\,=\, \det  R^{\bullet} f_*F
$$
over $S$ \cite{KM}. Given a
virtual vector bundle $E-F$ on $X$, where $E$ and $F$ are vector bundles
on $X$, there is the determinant line bundle ${\rm Det}(E-F)\,:=\,
{\rm Det}(E)\bigotimes {\rm Det}(F)^*$ on $S$. This determinant
${\rm Det}$ is a homomorphism from the $K$--group of vector bundles
on $X$ to the Picard group of $S$.

Let $L$ be a relatively very ample line bundle on $X$. Then for all
sufficiently large positive integers $m$, there is a natural isomorphism
of line bundles
\begin{equation}\label{vp}
\varphi\, :\,{\rm Det}((L- {\mathcal O}_{X})^{\otimes
(n+1)}\otimes L^{\otimes m})\, \longrightarrow\, \langle L\, ,
\cdots\, ,L\rangle
\end{equation}
\cite[Theorem 3]{PRS}.

Assuming that the above morphism $f$ is smooth, if we fix hermitian
structures on the above two vector bundles $E$ and $F$, and also fix a
relative K\"ahler form on $X$, then the line bundle ${\rm Det}(E-F)$
gets a hermitian structure \cite{BGS}, \cite{Qu}.

Our aim here is to prove the following (see Theorem~\ref{thm-m}):

\begin{theorem}\label{thm0}
Let $f\,:\,X \,\longrightarrow\, S$ be a smooth projective surjective morphism
between integral schemes over $\mathbb C$. Let $(L\, ,h)$ be a relatively very
ample hermitian line bundle on $X$. Fix a relative K\"ahler form on $X$, and
also equip ${\mathcal O}_{X}$ with the hermitian structure for which the constant
function $1$ has point-wise norm $1$. Let $h^Q$ and $h^D$ be the hermitian
structures on the line bundles ${\rm Det}((L- {\mathcal O}_{X})^{\otimes
(n+1)}\otimes L^{\otimes m})$ and $\langle L\, , \cdots\, ,L\rangle$
respectively. Then there is a positive real number $c$ such that
$$
\varphi^*h^D\, =\, c\cdot h^Q\, ,
$$
where $\varphi$ is the isomorphism in \eqref{vp}.

The constant $c$ is compatible with base change.

The constant $c$ is independent of the hermitian structure $h$.
\end{theorem}

It should be emphasized that the constant $c$ in Theorem~\ref{thm0} may
depend on the family $f$ and also on the N\'{e}ron--Severi class of the
restriction of $L$ to a fiber of $f$. However we do not know explicit examples.
It would be interesting to be able to construct explicit examples with different
values of $c$. It should be mentioned that
the proof of Theorem \ref{thm0} relies on curvature computations.
Therefore, the constant $c$ in Theorem \ref{thm0} cannot be computed using the
method of proof given here.

The construction of $h^D$ requires only a hermitian structure on $L$. In
particular, it
does not require to have a relative K\"ahler form on $X$. Although in general
the Quillen metric on a determinant line bundle depends on the choice of a
relative K\"ahler form, it turns out that the hermitian structure $h^Q$
on ${\rm Det}((L- {\mathcal O}_{X})^{\otimes (n+1)}\otimes L^{\otimes m})$
is independent of the choice of the relative K\"ahler form on $X$; see
Proposition~\ref{pr:kaehindep}.

The proof of Theorem~\ref{thm0} is based on the following:
\begin{itemize}
\item There are no nonconstant harmonic functions on a compact
    connected complex manifold.

\item The metric on a Deligne pairing of line bundles over a smooth
    variety has some regularity as the variety degenerates \cite{mo}.

\item The Quillen metric for hermitian vector bundles over a smooth
    variety can be controlled as the variety degenerates
    \cite{bismut}, \cite{yoshi}.
\end{itemize}

\section{Comparison of curvatures}\label{sec2}

\subsection{Deligne pairing}

Let $f\,:\,X \,\longrightarrow\, S$ be a smooth projective surjective morphism
between integral schemes over $\mathbb C$. The relative dimension of this
projection will be denoted by $n$. Fix $n+1$ line bundles $L_0,
\cdots ,L_n$ on $X$. The Deligne pairing $\langle L_0\, ,L_1\, ,\cdots
\, ,L_n\rangle$ is a line bundle on the base $S$ that defines a map
$$
(\mathcal{P}ic(X))^{n+1}\, \longrightarrow\, \mathcal{P}ic(S)\, ,
$$
where $\mathcal{P}ic$ denotes the stack of line bundles, which is multilinear
and invariant under the permutations of the factors, \cite{De}, \cite{Zh}.
We recall that the above the line bundle 
$\langle L_0\, ,L_1\, ,\cdots \, ,L_n\rangle$ is locally generated by symbols
of the form $\langle \sigma_0\, ,\sigma_1\, , \cdots\, ,\sigma_n\rangle$, where
$\sigma_i$, $0\, \leq\, i\, \leq\, n$, is a rational section
of the line bundle $L_i$ whose divisor ${\rm Div}(\sigma_i)$ has the property
that
$$
\bigcap_{i=0}^n {\rm Div}(\sigma_i)\,=\, \emptyset\, .
$$
The transition functions for the above local trivializations
satisfy the following condition: if
$$
{\mathcal D}\, :=\, \bigcap_{i\not=k} {\rm Div}(\sigma_i)
$$
is flat over $S$, then 
$$
\langle \sigma_0\, ,\cdots \, , \zeta\sigma_k\, ,\cdots\, ,\sigma_n\rangle\,
= \, \zeta[{\mathcal D}]\cdot \langle \sigma s_0\, ,\cdots \,,\sigma_n\rangle
$$
for any rational function $\zeta$, where 
and $\zeta[{\mathcal D}]\,=\,{\rm Norm}_{{\mathcal D}/S}(\zeta)$ is the
norm function. Using the two properties of Deligne pairing mentioned
earlier, it can be shown that this condition uniquely determines the
transition functions. The key input for this is the Weil reciprocity
formula.

For each $0\, \leq\, i\, \leq\, n$, fix a $C^\infty$ hermitian
structure $h_i$ on the line bundle $L_i$. Given this data, there is a
hermitian structure $\langle h_0\, ,\cdots \, ,h_n\rangle$
on the Deligne pairing $\langle L_0\, ,\cdots \, ,L_n\rangle$. This
hermitian structure $\langle h_0\, ,\cdots \, ,h_n\rangle$
satisfies the following condition: Take the hermitian structure
$h'_0\,:=\, h_0\exp(-\mu)$ on $L_0$, where $\mu$ is a $C^\infty$
real valued function on $X$. Then 
$$
\langle h'_0\, ,\cdots \, ,h_n\rangle\,=\, \langle h_0\, ,\cdots \,,h_n
\rangle\cdot \exp(-\widehat{\mu})\, ,
$$
where $\widehat{\mu}$ is the real valued smooth function
$$
\left(\frac{\sqrt{-1}}{2\pi}\right)^n \int_{X/S} \mu
\cdot(\partial{\overline\partial}\log h_1) 
\wedge\cdots \wedge (\partial{\overline\partial}\log h_n)\, .
$$

\subsection{The curvature}

Fix a relative K\"ahler form
on $X$. Fix a holomorphic hermitian line bundle $(L\, ,h)$ on $X$. Equip
${\mathcal O}_X$ with the standard hermitian structure
for which the constant function $1$ has point-wise norm $1$. Take any integer $m$.
Let $h^Q$ be the hermitian structure on ${\rm Det}((L- {\mathcal
O}_X)^{\otimes (n+1)}\otimes L^{\otimes m})$. The hermitian structure on
$\langle L\, , \cdots\, ,L\rangle$ will be denoted by $h^D$.

\begin{proposition}\label{thm2}
The curvature of the hermitian metric $h^D$ on $\langle L\, , \cdots\,
,L\rangle$ coincides with the curvature of the Quillen metric $h^Q$ on
${\rm Det}((L- {\mathcal O}_X)^{\otimes (n+1)}\otimes L^{\otimes m})$.
\end{proposition}

\begin{proof}
The Chern form of the metric $h^D$ on $\langle L,\cdots,L\rangle$
coincides with the fiber integral
\begin{equation}\label{del}
\int_{X/S} c_1(L,h)^{n+1}
\end{equation}
(see \cite[p. 81, Section 1.2]{Zh}). On the other hand, a theorem of
Bismut, Gillet and Soul\'e, \cite[Theorem 0.1]{BGS}, says that the Chern
form of the determinant line bundle
$$
({\rm Det}((L- {\mathcal
O}_X)^{\otimes (n+1)}\otimes L^{\otimes m})\, , h^Q)\, \longrightarrow\,
S
$$
is the degree two component of the Riemann-Roch fiber integral
\begin{equation}\label{ht2}
c_1({\rm Det}((L- {\mathcal O}_{X})^{\otimes (n+1)}\otimes L^{\otimes m}
,h^Q)\,=\, \Big(\int_{X/S}ch(L-\cO_X,h)^{n+1}\cdot ch(L,h)^m\cdot
Td(X/S)\Big)_{(2)}\, ,
\end{equation}
where $Td$ is the Todd form. This result of \cite{BGS} was extended to
(smooth) K\"ahler fibrations over singular base spaces in
\cite[\S~12]{f-s}. Since the base $S$ is reduced and irreducible, we may
choose a desingularization $S'$ of $S$ and consider the pulled back family
$X\times_S S'$ over $S'$. Let $(L'\, ,h')$ be the pullback of
$(L\, ,h)$ to $X\times_S S'$.
Then \eqref{ht2} for $L'$ coincides with the pullback of \eqref{ht2} to $S'$.

Note that
$$
ch(L-\cO_X, h) \,=\, c_1(L,h) + ~\text{~higher~order~terms}\, .
$$
Hence the only contribution of the differential form $ch(L,h)^m\cdot
Td(X/S)$ in the integral in \eqref{ht2} is the constant $1$, and also the
higher order terms in $ch(L-\cO_X, h)$ do not contribute in the integral.
Consequently, the integral in \eqref{ht2} coincides with the one in
\eqref{del}.
\end{proof}

\section{Independence from the relative K\"ahler
structure}\label{sec2a}

Let $f\,:\, X \,\longrightarrow\, S$ be a smooth projective surjective
morphism of relative dimension $n$, where $X$ and $S$ are integral schemes
over $\mathbb C$. Fix a relative K\"ahler form $\omega_{\cX/S}$ on $X$.
Let $L$ be a line bundle on $X$ equipped with a hermitian structure. As
before, fix the hermitian metric on ${\mathcal O}_X$ for which the constant
function $1$ has point-wise norm $1$. Let $h^Q$
be Quillen metric on ${\rm Det}( (L- {\mathcal O}_{X})^{\otimes
(n+1)}\otimes L^{\otimes m})$.

\begin{proposition}\label{pr:kaehindep}
The Quillen metric $h^Q$ on ${\rm Det}((L- {\mathcal O}_{X})^{\otimes
(n+1)}\otimes L^{\otimes m})$ does not depend on the relative K\"ahler
form $\omega_{\cX/S}$.
\end{proposition}

\begin{proof}
In view of the base change property of $h^Q$, it suffices to prove the
proposition for the absolute case, meaning when $S$ is a single point.

So, let $(X\, ,\omega_X)$ be a compact connected K\"ahler manifold, and let
$(L\, ,h)$ be a holomorphic hermitian line bundle on it. Let $\xi$ denote the
virtual vector bundle $(L - \cO_X)^{\otimes (n+1)}\otimes L^{\otimes m}$
equipped with the (virtual) hermitian structure $h_\xi$ given by $h$ and
the hermitian metric on ${\mathcal O}_X$ for which the constant
function $1$ has point-wise norm $1$. Then the Chern
character form for $h_\xi$ is
\begin{equation}\label{eq:chL}
ch(\xi,h_\xi)\,=\,  c_1(L,h)^{n+1} +
\text{higher~order~terms.}
\end{equation}

Let $\omega'_X$ be another K\"ahler form on $X$. Bott and Chern, \cite{BC},
defined a class of differential forms $\wt{Td}(\omega_X,\omega'_X)$,
unique in the space of forms modulo the images of $\pt_X$ and $\ol\pt_X$,
such that
$$
\frac{1}{2\pi\sqrt{-1}} \pt_X\ol\pt_X \wt{Td}(\omega_X,\omega'_X) \,=\,
Td(X,\omega_X)-Td(X,\omega'_X)\, .
$$

Let $h^Q$ (respectively, $h^{'Q}$) be the Quillen metric on the
determinant line bundle (now it is just a complex line as $S$ is
only a point) ${\rm
Det}((L- {\mathcal O}_{X})^{\otimes n}\otimes L^{\otimes m})$ for the K\"ahler
metric $\omega_X$ (respectively, $\omega'_X$). A theorem of Bismut, Gillet
and Soul\'e, \cite[Theorem~0.2]{BGS}, says that
\begin{equation}\label{eq:bgs02}
\log\left(\frac{h^Q}{h^{'Q}}\right)\, =\,
\int_X\wt{Td}(\omega_X,\omega'_X)ch(\xi,h_\xi)\, .
\end{equation}
Only the term in degree $2n$ of the above integrand contributes, and by
\eqref{eq:chL}, this term vanishes. Therefore,
$$
\log\left(\frac{h^Q}{h^{'Q}}\right)\, =\, 0\, .
$$
This implies that $h^Q\,=\, h^{'Q}$. This completes the proof of the
lemma.
\end{proof}

We fix a complex projective manifold $X$. Let $\text{NS}(X)$ be the N\'eron--Severi
group $H^2(X,\, {\mathbb Z})\bigcap H^{1,1}(X)$ of $X$.
For an element $\nu\,\in\, \text{NS}(X)$ of $X$, let $P\,=\,{\rm Pic}^{\nu}(X)$
be the connected component of the Picard group of $X$ corresponding to $\nu$.
We choose $\nu$ such that all line bundles on $X$ lying in $P$ are very ample.
Take any $L\, \in\, P$, and take any hermitian
structure $h_L$ on $L$. Let $h^Q_L$ and $h^D_L$ be the hermitian
structures on ${\rm Det}((L- {\mathcal O}_{X})^{\otimes (n+1)}\otimes
L^{\otimes m})$ and $\langle L\, ,\cdots\, ,L\rangle$ respectively.

\begin{proposition}\label{pr:pic}
There exists a real number $c_\nu\, >\, 0$, that depends only on $X$ and
$\nu\,\in\, {\rm NS}(X)$, such that for any $L$ and $h_L$ as above, the
following equality of hermitian structures on ${\rm Det}((L- {\mathcal O}_{X})^{
\otimes (n+1)}\otimes L^{\otimes m})$ holds:
\begin{equation}
c_\nu\cdot h^Q_L\,=\, \varphi^* h^D_L\, ,
\end{equation}
where $\varphi$ is the isomorphism in \eqref{vp}. In particular, $c_\nu$
is independent of the point $L\, \in\, {\rm Pic}^{\nu}(X)$ and the hermitian
structure $h_L$.
\end{proposition}

\begin{proof}
Fix a Poincar\'e line bundle ${\mathcal L}\, \longrightarrow\, X\times P$.
Let $h$ be a hermitian structure on $\mathcal L$. Consider the trivial
family $X\times P \,\longrightarrow\, P$. Let
$$
{\rm Det}(({\mathcal L}- {\mathcal
O}_{X\times P})^{\otimes (n+1)}\otimes {\mathcal L}^{\otimes m})
\, \longrightarrow\,  P
$$
be the determinant line bundle. Let $h^Q$ and $h^D$ be the hermitian
structures on the line bundles ${\rm Det}(({\mathcal L}- {\mathcal
O}_{X\times P})^{\otimes (n+1)}\otimes {\mathcal L}^{\otimes m})$ and
$\langle {\mathcal L}\, ,\cdots\, , {\mathcal L}\rangle$ respectively. The
curvatures of $h^Q$ and $h^D$ coincide by Proposition~\ref{thm2}.
Consequently,
$$
\log\left(\frac{\varphi^* h^D}{h^{Q}}\right)
$$
is a harmonic function on $P$, where $\varphi$ is the isomorphism in
\eqref{vp}. On the other hand, there are no nonconstant harmonic functions
on $P$ because it is compact and connected. Therefore, $(\varphi^* h^D)/h^Q$ is a
constant function on $P$.

For any point $L\, \in\, P$, any hermitian structure on ${\mathcal
L}\vert_{X\times\{L\}}\,=\, L$ can be extended to a hermitian structure on
$\mathcal L$. Therefore, to prove the proposition, it suffices to show
that $(\varphi^* h^D_L)/h^Q_L$ is independent of the hermitian structure
on $L$.

To prove that $(\varphi^* h^D_L)/h^Q_L$ is independent of the hermitian
structure on $L$, take any two hermitian structures $h^1_L$ and $h^2_L$ on
$L$. Let $Z$ be a compact connected complex manifold of positive
dimension. Fix two distinct points $z_0$ and $z_1$ of $Z$. Consider the
line bundle
$$
p^*_X L\, \longrightarrow\, X\times Z\, ,
$$
where $p_X$ is the natural projection of $X\times Z$ to $X$. Let $h_1$ and $h_2$
be two hermitian structures on $p^*_X L$ such that
\begin{enumerate}
\item the restriction of $h_1$ (respectively $h_2$) to $(p^*_X
    L)\vert_{X\times\{z_0\}}$ is $h^1_L$ (respectively $h^2_L$), and

\item the restrictions of $h_1$ and $h_2$ to $(p^*_X
    L)\vert_{X\times\{z_1\}}$ coincide.
\end{enumerate}
Let $h^Q_1$ (respectively, $h^Q_2$) be the hermitian metric on
$$
{\rm Det}((p^*_X L- {\mathcal
O}_{X\times Z})^{\otimes (n+1)}\otimes (p^*_X L)^{\otimes m})
\, \longrightarrow\, Z
$$
for $h_1$ (respectively $h_2$). Similarly, let $h^D_1$ (respectively, $h^D_2$) be
the hermitian metric on
$$
\langle p^*_X L\, ,\cdots\, , p^*_X L\rangle\, \longrightarrow\, Z
$$
for $h_1$ (respectively $h_2$). As before, there is a constant $d_1$
(respectively, $d_2$) such that
$$
d_1\cdot h^Q_1\,=\, \varphi^* h^D_1 ~\, \text{ (respectively,~}
d_2\cdot h^Q_2\,=\, \varphi^* h^D_2{\rm )}\, ,
$$
because there are no nonconstant harmonic functions on the compact connected
complex manifold $Z$ (the
isomorphism $\varphi$ is as in \eqref{vp}). Note that we are using that the
curvature of $h^Q_1$ (respectively, $h^Q_2$) coincides with that of
$h^D_1$ (respectively, $h^D_2$) by Proposition~\ref{thm2}. Since the
restriction of $h_1$ to $(p^*_X L)\vert_{X\times\{z_1\}}$ coincides, by
construction, with the restriction of $h_2$ to $(p^*_X
L)\vert_{X\times\{z_1\}}$, we now conclude that $d_1\,=\, d_2$ using base
change.
\end{proof}

Henceforth, we will identify the two sides of \eqref{vp} using the
isomorphism $\varphi$. In particular, we will omit referring explicitly to
$\varphi$.

\section{Comparison of metrics: the case of Riemann
surfaces}\label{sec3}

The above argument that uses Proposition~\ref{thm2} combined with the fact that
there are no nonconstant harmonic functions on a compact complex manifold,
can be extended to line bundles on families of Riemann surfaces.

\begin{proposition}\label{pc}
Let $X$ be a compact connected Riemann surface of genus $p$ with $p \,
\geq\, 3$, and let $L$ be a holomorphic hermitian line bundle on $X$ of
degree $d$, with $d\, \geq\, 2p-1$. Let $h^Q$ and $h^D$ be the hermitian
structures on ${\rm Det}((L- {\mathcal O}_{X})^{\otimes 2}\otimes
L^{\otimes m})$ and $\langle L\, ,L\rangle$ respectively. Then the
positive real number $c$ that satisfies the identity
$$
c\cdot h^Q\, =\, h^D
$$
depends only on $p$ and $d$. In other words, $c$ is independent of both
the conformal structure of $X$ and the holomorphic structure of $L$.
\end{proposition}

\begin{proof}
The above assumption that $d\, \geq\, 2p-1$ is needed to ensure that
the line bundles are very ample. This enables us to invoke \cite[Theorem 3]{PRS}.

{}From Proposition~\ref{pr:pic} and Proposition~\ref{pr:kaehindep} we know
the constant $c$ in the proposition depends only on the degree $d$ and the
complex structure of the Riemann surface, and it is independent of the metric on
$X$.

Let $\cM_p$ be the moduli space of compact Riemann surfaces of genus $p$.
Let $\widetilde{\cM}_p$ be the Satake compactification of $\cM_p$. We
recall that $\widetilde{\cM}_p$ is defined to be the closure of $\cM_p$ in
the Satake compactification of the moduli space of principally polarized
Abelian varieties \cite{Ba}. The Satake compactification
$\widetilde{\cM}_p$ is an irreducible projective variety, and the complement
$$
\widetilde{\cM}_p\setminus \cM_p\, \subset\, \widetilde{\cM}_p
$$
is of codimension at least two \cite{Ba}, \cite[p. 45]{HM}. This codimension is one
if $p\,=\, 2$. The assumption in the proposition that $p\, \geq\, 3$ is needed here.

Consequently, there is a dense subset $U\, \,\subset\, \cM_p$ in Euclidean
topology satisfying the following condition: for any two points $x\, ,y\,
\in\, U$, there are compact Riemann surfaces $C_1\, ,\cdots\, , C_\ell$,
and holomorphic maps
$$
\phi_i\, : C_i\, \longrightarrow\, \cM_p\, ,
$$
such that
\begin{itemize}
\item $\phi_i(C_i)\, \subset\, U$ for all $1\,\leq\, i\,\leq\, \ell$,

\item $\bigcup_{i=1}^\ell \phi_i(C_i)$ is connected, and

\item $\{x\, ,y\}\, \subset\, \bigcup_{i=1}^\ell \phi_i(C_i)$.
\end{itemize}
The constant $c$ is clearly continuous over $\cM_p$.
Therefore, to prove the proposition, it is enough to show that for any
holomorphic map
$$
\phi\, :\, C\, \longrightarrow\, \cM_p\, ,
$$
where $C$ is some compact connected Riemann surface, the constant $c$ in the
proposition remains unchanged as the Riemann surface $X$ moves over
$\phi(C)\,\subset\, \cM_p$ (with $d$ fixed).

Take $(C\, ,\phi)$ as above. There is a finite (possibly ramified) covering
$$
\gamma\, :\, \widetilde{C}\, \longrightarrow\, C
$$
such that there is an algebraic family of Riemann surfaces over
$\widetilde{C}$ represented by the morphism
$$
\phi\circ\gamma\, :\,\widetilde{C}\, \longrightarrow\, \cM_p\, .
$$
In other words,
$\phi\circ\gamma$ is the classifying map for this family. It may be
mentioned that we may take
$\widetilde{C}$ to be the parameter space for Riemann surfaces corresponding
to points in $\phi(C)$ equipped with a level structure.

In view of Proposition~\ref{thm2}, we conclude that the quotient $\log
(h^Q/h^D)$ is constant for the above family over $\widetilde{C}$, because
there are no nonconstant harmonic functions on $\widetilde{C}$. Therefore,
$h^Q/h^D$  is a constant. As explained above, this completes the proof of
the proposition.
\end{proof}

\section{Comparison of metrics: The general case}

The following is the key lemma.

\begin{lemma}\label{le:singcurv}
Let \/$\overline C$ be a smooth complete connected complex curve, and let
$C\,\subset\, {\overline C}$ be the complement finitely many points. Let
$f\,:\, {\overline \cX}\,\longrightarrow\,{\overline C}$ be a flat
projective family of $n$-dimensional varieties, which is smooth with
connected fibers over $C$. Let $\overline{\cL}$ be a relatively very ample
line bundle over ${\overline \cX}$ equipped with a hermitian structure.
Define ${\cX}\, :=\, f^{-1}(C)$, and ${\cL}\, :=\,
\overline{\cL}\vert_{\cX}$. Then there is a positive real number $c$ such
that
$$
c\cdot h^Q\,=\, h^D\, ,
$$
where $h^Q$ is the Quillen metric on the line bundle ${\rm Det}(({\cL}-
{\mathcal O}_{\cX})^{\otimes (n+1)}\otimes ({\cL})^{\otimes m})$ and $h^D$
is the metric on the Deligne pairing $\langle {\cL}\, ,\cdots\,
,{\cL}\rangle$ (recall that the line bundles are identified by \eqref{vp}).
\end{lemma}

\begin{proof}
From Proposition~\ref{thm2} we know that $\log(h^Q/h^D)$ is harmonic on
$C$. A theorem of Moriwaki says that the hermitian structure $h^D$ possesses
a continuous extension to the line bundle
$$
\langle \overline{\cL}\, ,\cdots\, ,\overline{\cL}\rangle\,
\longrightarrow\, \overline{C}
$$
(see \cite[Theorem A]{mo}).

The degeneration of the Quillen metric for semistable families was
first computed by Bismut in \cite[(0.7)]{bismut}, and this was generalized
by Yoshikawa in \cite[Theorem 1.1]{yoshi}. In order to apply this
result, first the  semistable reduction theorem of Mumford,
\cite[p.~54]{kkms}, is used. It yields the existence of a finite morphism
$$
\nu\,:\,{\wt C} \,\longrightarrow\, \overline{C}
$$
with the following property: The desingularization
$$
Z \,\longrightarrow\, \ol\cX\times_{\overline C}\wt C
$$
gives rise to a family $Z \,\longrightarrow\,\wt C$, which is smooth over
$\nu^{-1}(C)$ and whose fibers over points of $\nu^{-1}(\ol C\backslash
C)$ are simple normal crossings divisors. The projective variety $Z$ is
now equipped with the restriction of a Fubini--Study metric (with respect to some
projective embedding).

Let
$$
\phi\, :\, Z\longrightarrow\, {\wt C}
$$
be the natural projection. Let
\begin{equation}\label{eta}
\eta\, :\, Z\, \longrightarrow\, \ol\cX
\end{equation}
be the obvious morphism.

According to \cite[Theorem~0.1, (0.6)]{bismut}, the Chern form of the
Quillen metric over $C$ extends as the sum of two currents on $\widetilde
C$. The first current is in $L^r_{loc}({\widetilde C})$ with $1\,\leq\, r \,<\,2$,
and the second one is equal to
\begin{equation}\label{int}
-\frac{1}{2} \left[\int_\Sigma Td(T(\Sigma/{\widetilde
C}))E(\alpha)
ch(\xi)\right]^{(0)} \delta_\Delta\, ,
\end{equation}
where
\begin{itemize}
\item $\Delta\,:=\, (\nu\circ\phi)^{-1}(\overline{C}\setminus C)$ is
    the singular fiber,

\item $\Sigma\subset \Delta$ is the singular locus of $\Delta$,

\item $\alpha$ is the normal bundle of $\Sigma$ in $\Delta$,

\item $\xi$ is the virtual vector bundle $(\eta^*\overline{\mathcal L}
    - \cO_Z)^{\otimes (n+1)}\otimes (\eta^*\overline{\mathcal
L})^{\otimes m}$ equipped with the (virtual) hermitian structure
    $h_\xi$, defined by the hermitian metric on $\overline{\mathcal
    L}$ and the standard hermitian metric on $\cO_Z$, and

\item $E$ is the additive arithmetic genus that is generated by a
    certain holomorphic function on $\C$.
\end{itemize}

By \eqref{eq:chL}, the integral in \eqref{int} vanishes. Take any boundary
point $z_0\, \in\, \nu^{-1}(\overline{C}\setminus C)$; fix a local
holomorphic coordinate $s$ on an open neighborhood $U_{z_0}\, \subset\,
\widetilde{C}$ of $z_0$ such that $s(z_0)\,=\,0$. Since the function
$\log(h^Q/h^D)\circ \nu$ on $U_{z_0}\setminus\{z_0\}$ is  harmonic,
it is of the form
$$
a+ \gamma \log |s|\cdot {\rm Re}(\psi)\, ,
$$
where $a$ and $\gamma$ are constants, and $\psi$ is a holomorphic function
on $U_{z_0}\setminus\{z_0\}$. Since the Chern current of $h^Q$ is in
$L^1_{loc}$,  we conclude that
\begin{enumerate}
\item $\gamma\,=\,0$, and

\item $\psi$ extends holomorphically across $z_0$.
\end{enumerate}
Note that $h^D$, considered as function after choosing a trivialization of
$\langle \overline{\cL}\, ,\cdots\, ,\overline{\cL}\rangle$ over $U_{z_0}$,
is continuous and nowhere vanishing on $U_{z_0}$.

Consequently, $\log(h^Q/h^D)\circ \nu$ extends as a harmonic function from
$\nu^{-1}(C)$ to $\widetilde C$. Therefore, $\log(h^Q/h^D)$ is a constant
function.

However, the above theorem of Bismut applies to singular fibers, where at
most two components meet at any point. The methods of Bismut were extended
to the general case of arbitrary singular fibers of flat projective
morphisms by Yoshikawa in \cite{yoshi}. Yoshikawa's main theorem,
\cite[Theorem~1.1]{yoshi}, implies that the Quillen metric extends
continuously up to a certain summand, which consist of an integral
containing the Chern character form $ch(\xi, h_\xi)$ over an
$n$--dimensional variety. Again the integral vanishes by \eqref{eq:chL}.
This shows that both metrics $h^Q$ and $h^D$ on the determinant line
bundle extend continuously to $\widetilde C$; here we are using that the
line bundles
$$
{\rm Det}((\eta^*{\overline L}- {\mathcal
O}_{Z})^{\otimes (n+1)}\otimes (\eta^*{\overline L})^{\otimes m})
\, \longrightarrow\, \widetilde{C} ~\, \text{~and~}\,
\langle \eta^*{\overline L}\, ,\cdots\, ,\eta^*{\overline
L}\rangle\, \longrightarrow\,\widetilde{C}
$$
are canonically isomorphic (see \eqref{vp}, \cite[Theorem 3]{PRS}).
Consequently, the quotient $\log(h^Q/h^D)\circ\nu$ is a continuous
harmonic function of $\widetilde{C}$, implying that $\log(h^Q/h^D)$ is a
constant.
\end{proof}

The following theorem was mentioned in the introduction (Theorem
\ref{thm0}) as the main result.

\begin{theorem}\label{thm-m}
Let $f\,:\,X \,\longrightarrow\, S$ be a smooth projective surjective morphism
between integral schemes over $\mathbb C$. Let $(L\, ,h)$ be a relatively
very ample hermitian line bundle on $X$.
Fix a relative K\"ahler form on $X$, and also equip ${\mathcal O}_{X}$ with
the natural hermitian structure. Let $h^Q$ and $h^D$ be the
hermitian structures on the line bundles ${\rm Det}((L- {\mathcal
O}_{X})^{\otimes (n+1)}\otimes L^{\otimes m})$ and $\langle L\, , \cdots\,
,L\rangle$ respectively. Then there is a positive real number $c$ such
that
$$
\varphi^*h^D\, =\, c\cdot h^Q\, ,
$$
where $\varphi$ is the isomorphism in \eqref{vp}.

The constant $c$ is compatible with base change.

The constant $c$ is independent of the hermitian structure $h$.
\end{theorem}

\begin{proof}
In view of the base change property of the isomorphism in \eqref{vp}, the
first part of the theorem follows from Lemma~\ref{le:singcurv}. To explain this,
take any morphism
$$
\rho\, :\, C\,\longrightarrow\, S\, ,
$$
where $C$ is a smooth complex curve such that
\begin{itemize}
\item the pullback $C\times_{S} X$
extends to a flat family of projective varieties over the smooth compactification
$\overline{C}$ of $C$, and

\item the pullback of $L$ to $C\times_{S} X$ extends to a relative very ample line
bundle for the above family over $\overline{C}$.
\end{itemize}
{}From Lemma~\ref{le:singcurv} and the
base change property of the isomorphism in \eqref{vp} we have the following.
There is a positive real number $c_\rho$ such that for any point $z\,\in\, C$, the equality
$$
(\varphi^*h^D)_{\rho(z)} \, =\, c_\rho \cdot (h^Q)_{\rho(z)}
$$
holds. Now we define an equivalence relation on the points on $S$. Two points
$z$ and $z'$ of $S$ will be called equivalent if there are finitely many morphisms
$\rho_i$ of the above type such that the union $\bigcup_i \text{image}(\rho_i)$ is
connected and contains both $z$ and $z'$. From the above observation and
the base change property of the isomorphism in \eqref{vp} we conclude that for any
equivalence class $\mathcal C$, there is a positive real number $c_{\mathcal C}$
such that for every point $z\,\in\, c_{\mathcal C}$, we have
$$
(\varphi^*h^D)_z \, =\, c_{\mathcal C} \cdot (h^Q)_z\, .
$$
Since there exists an equivalence class which is dense in $S$ in the Euclidean
topology, the first part of the theorem follows.

The determinant line bundle, the Quillen metric, the Deligne pairing and
the hermitian structure on the Deligne pairing are all compatible with
base change. The isomorphism $\varphi$ in \eqref{vp} is also compatible
with base change. Therefore, the second part of the theorem follows.

The third part follows from Proposition~\ref{pr:pic}.
\end{proof}

\section{Further application of the method}\label{sec-l}

We recall Theorem 5.8 in page 371 of \cite{ACG}:

\begin{theorem}[\cite{ACG}]\label{t-a}
Let $f\, :\, X\, \longrightarrow\, S$ be an algebraic family of nodal
complex projective curves. Let $L$ and $M$ be two algebraic line bundles
on $X$. There is a canonical isomorphism
$$
\langle L\, ,M\rangle \, \stackrel{\sim}{\longrightarrow}\,
{\rm Det}(L\otimes M)\otimes {\rm Det}(L)^*\otimes {\rm Det}(M)^*
\otimes {\rm Det}({\mathcal O}_X)\,=\, {\rm Det}((L-{\mathcal
O}_X)(M-{\mathcal O}_X))
$$
compatible with base change.
\end{theorem}

Theorem~\ref{t-a} generalizes to higher dimensions. To explain this, let
$X\, \longrightarrow\, S$ be a flat projective surjective morphism, of
relative dimension $n$, between integral schemes over $\mathbb C$. Fix
line bundles $L_0\, , L_1\, ,\cdots\, , L_n$ on $X$.

\begin{theorem}[\cite{BSW}]\label{t-bsw}
There is a canonical isomorphism of the Deligne pairing
$$
\langle L_0\, ,\cdots\, ,L_n\rangle\, \longrightarrow\, S
$$
and the determinant line bundle ${\rm Det}(\bigotimes_{i=0}^n
(L_i-{\mathcal O}_X))$, which is compatible with base change.
\end{theorem}

For each $0\, \leq\, i\,\leq\, n$, fix a hermitian structure $h_i$ on $L_i$.
This produces a hermitian structure $h^D$ on the line bundle $\langle
L_0\, ,\cdots\, ,L_n\rangle\, \longrightarrow\, S$. Now fix a relative
K\"ahler form $\omega_{X/S}$ on $X$. This and the hermitian metrics
$h_i$ together define a hermitian structure $h^Q$ on the line bundle ${\rm
Det}(\bigotimes_{i=0}^n (L_i- {\mathcal O}_X))$ \cite{Qu}, \cite{BGS}.

It can be shown that $h^Q$ is independent of the choice of the relative
K\"ahler form $\omega_{X/S}$. Indeed, its proof is exactly identical to that of
Proposition~\ref{pr:kaehindep}. In fact the following holds:

\begin{theorem}\label{thml}
There is a positive real number $c$ such that the isomorphism in
Theorem~\ref{t-bsw} takes the hermitian structure $h^D$ to $c\cdot h^Q$.
\end{theorem}

The proof of Theorem~\ref{thml} is same as that of Theorem~\ref{thm-m}.

\section*{Acknowledgements}

The first--named author acknowledges the
support of the J. C. Bose Fellowship.


\end{document}